\documentclass{amsart}

\usepackage{epsfig}
\usepackage{amsmath}
\usepackage{amssymb}
\usepackage{amscd}
\usepackage{graphicx}
\usepackage{color}
\usepackage{verbatim}

\newtheorem{thm}{Theorem}[section]
\newtheorem{lem}[thm]{Lemma}
\newtheorem{prop}[thm]{Proposition}
\newtheorem{fact}[thm]{Fact}

\newtheorem{que}[thm]{Question}

\newtheorem{defn}[thm]{Definition}

\theoremstyle{remark}
\newtheorem{rem}[thm]{Remark}

\newtheorem{exam}[thm]{Example}

\def \N {\mathbb N}

\def \K {\mathcal K}
\def \Z {\mathbb Z}

\def \E {\mathbb E}

\def \corr {\mathsf{corr}}
\def \sq {sequence}
\def \xt {$(X,T)$}

\def \xmt {$(X,\mathcal A,\mu,T)$}
\def \xmtp {$(X',\mathcal A',\mu',T')$}
\def \yns {$(Y,\mathcal B,\nu,S)$}
\def \tl {topological}
\def \im {invariant measure}

\def \ds {dynamical system}

\def \mmu {\boldsymbol{\mu}}

\numberwithin{equation}{section}

\begin{document}

\title[Correlation of sequences and of measures]{Correlation of sequences and of measures, generic points for joinings and ergodicity of certain cocycles}

\author{Jean-Pierre Conze, Tomasz Downarowicz and Jacek Serafin}

\address{IRMAR, CNRS UMR 6625, University of Rennes I, Campus de Beaulieu, 35042 Rennes Cedex, France}

\address{Institute of Mathematics, Polish Academy of Science, \'Sniadeckich 8, 00-656 Warsaw, Poland}

\address{Institute of Mathematics and Computer Science, Wroc{\l}aw University of Technology, Wroc{\l}aw 50-370, Poland}

\begin{abstract}
The main subject of the paper, motivated by a question raised by Boshernitzan, is to give criteria for a bounded complex-valued \sq\ to be uncorrelated to any strictly ergodic \sq. As a tool developed to study this problem we introduce the notion of correlation between two shift-\im s supported by the symbolic space with complex symbols. We also prove a ``lifting lemma'' for generic points: given a joining $\xi$ of two shift-\im s $\mu$ and $\nu$, every point $x$ generic for $\mu$ lifts to a pair $(x,y)$ generic for $\xi$ (such $y$ exists in the full symbolic space). This lemma allows us to translate correlation between bounded \sq s to the language of correlation of measures. Finally, to establish that the property of an \im\ being uncorrelated to any ergodic measure is essentially weaker than the property of being disjoint from any ergodic measure, we develop and apply criteria for ergodicity of four-jump cocycles over irrational rotations. We believe that apart from the applications to studying the notion of correlation, the two developed tools: the lifting lemma and the criteria for ergodicity of four-jump cocycles, are of independent interest. This is why we announce them also in the title. In the Appendix we also 
introduce the notion of conditional disjointness.
\end{abstract}

\maketitle
\thanks{The research of the second and third author is supported by the NCN (National Science Center, Poland) grant 2013/08/A/ST1/00275.}

\section{Basic notions and motivation}

\begin{defn}\label{defn2}
For \emph{finite} \sq s (blocks) $A=[a_1,a_2,\dots, a_n],\ B=[b_1,b_2,\dots, b_n]$ (of the same length $n$), consisting of complex numbers, we define
$$
\corr(A,B)=\Bigl(\frac1n\sum_{i=1}^na_i\overline{b_i}\Bigr)-\Bigl(\frac1n\sum_{i=1}^na_i\Bigr)\cdot\Bigl(\frac1n\sum_{i=1}^n\overline{b_i}\Bigr).
$$
Two complex-valued bounded \emph{infinite} \sq s $x, y$ are declared \emph{uncorrelated}, \emph{weakly correlated}, or \emph{strongly correlated}, if
\begin{align*}
\lim_n\corr(x_{[1,n]},y_{[1,n]})&=0, \\
\limsup_n|\corr(x_{[1,n]},y_{[1,n]})|&>0, \\
\liminf_n|\corr(x_{[1,n]},y_{[1,n]})|&>0,
\end{align*}
respectively. (We are using the following notation: $x_{[1,n]}$ is the block $[x_1,x_2,\dots,x_n]$.)
\end{defn}

Notice that if at least one sequence has zero mean, then the correlation uses only the first average
(of the products).

\medskip
Although most of our results apply to sequences in the complex space $\ell^\infty$, for simplicity of the forthcoming arguments we will restrict our attention to the space $K^\N$ of sequences taking values in a finite subset $K$ of the unit disc. This will allow us to use a range of measure-theoretic and \tl\ tools applicable for subshifts. We will call $K$ the \emph{alphabet}, and $\N$ denotes the set of positive integers (so that the enumeration of our \sq s always starts with the index 1).
\medskip

The famous \emph{Sarnak Conjecture} (see \cite{Sar}) asserts that the \emph{M\"obius function $\mmu$} is uncorrelated to any \emph{deterministic \sq\ $x$} where $\mmu$ is the ``signed characteristic function'' of square-free numbers:
$$
\mmu(n)=\begin{cases}
\phantom{-}1&\text{if $n=1$,}\\
\phantom{-}0& \text{if $n$ has a repeated prime factor,}\\
\phantom{-}(-1)^r&\text{if $n$ is a product of $r$ distinct primes},
\end{cases}
$$
and a deterministic \sq\ is any \sq\ $x$ of the form
$$
x_n = f(T^na),
$$
where $f$ is a continuous complex-valued function defined on a \tl\ \ds\ \xt\ (i.e., $T:X\to X$ is a
continuous transformation of a compact metric space $X$) of \tl\ entropy zero and $a\in X$ (equivalently, the shift orbit-closure of $x$ has \tl\ entropy zero). Observe that the set of square-free numbers has positive density in $\N$, so the conjecture is not trivial. An up-to-date exposition of classes of systems for which the conjecture holds, as well as some new results, can be found in \cite{Lem}. 
While Sarnak Conjecture is still far from being solved, our research is motivated by a similar problem (attributed to Boshernitzan), stated below. A \sq\ $x$ is \emph{strictly ergodic} if it is of the form
$$
x_n = f(T^na),
$$ 
where $f$ is a continuous complex-valued function defined on a strictly ergodic \tl\ \ds\ \xt\ and $a\in X$
(equivalently, the shift orbit-closure of $x$ is a strictly ergodic dynamical system).

\begin{que}\label{que1}
Is the M\"obius function uncorrelated to any strictly ergodic \sq ?
\end{que}

We have tried to understand what kind of question this is and what it means for a \sq\ to be uncorrelated to any strictly ergodic \sq. As it turns out, the condition defines an interesting, nontrivial class of \sq s.
\medskip

Let us now recall the notion of a generic point. By \emph{\im s} we will always mean $T$-invariant Borel probability measures on $X$.

\begin{defn}\label{generating}
A point $x$ in a \tl\ \ds\ \xt\ is \emph{generic} for an \im\ $\mu$ if 
$$
\lim_n\frac1n\sum_{i=1}^n f(T^i x) = \int f\,d\mu,
$$
for every continuous (real or complex) function $f$ on $X$.
\end{defn}

We recall the well-known fact that generic points always exist for ergodic measures, in which case they 
form a set of full measure. Moreover, if a \tl\ \ds\ is uniquely ergodic then all of its points are generic for the unique \im. It is also well known that in the full shift over a finite alphabet every shift-\im\ has a generic point. An important simplification which arises for symbolic systems is that for genericity of a point it suffices to verify the convergence for the countable family of characteristic functions of cylinder sets. Moreover, each point $x$ is now a \sq\ and the ergodic average of length $n$ for the characteristic function of a cylinder associated to a block $B$ corresponds to the \emph{frequency} with which $B$ appears in the initial block $x_{[1,n+|B|-1]}$. Thus $x$ is generic for a shift-\im\ $\mu$ if 
and only if the following holds, for every positive integer $m$ and every block $B\in K^m$:
\begin{equation}
\mu(B) = \lim_n \tfrac1n\#\{i\in\{1,2,\dots,n\}: x_{[i,i+m-1]}=B\}
\end{equation}
(by writing $\mu(B)$ we identify blocks with their associated cylinder sets in $K^\N$). 
We will need a similar notion of semi-generating a measure along a sub\sq:

\begin{defn}\label{semigen}
We say that a \sq\ $x\in K^\N$ \emph{semi-generates} a shift-\im\ $\mu$ \emph{along a sub\sq} $(n_k)$ if the following holds, for every positive integer $m$ and every block $B\in K^m$:
\begin{equation}\label{gen}
\mu(B) = \lim_k \tfrac1{n_k}\#\{i\in\{1,2,\dots,n_k\}: x_{[i,i+m-1]}=B\}.
\end{equation}
\end{defn}

By the weak-star compactness of the space of probability measures, given a point $x\in K^\N$, every sub\sq\ $(n_k)$ contains a sub-sub\sq\ along which $x$ semi-generates an \im. Now, given two points, $x$ and $y$, we can treat the pair $(x,y)$ as an element of $(K\times K)^\N$ (which is also a symbolic space), hence every sub\sq\ $(n_k)$ contains a sub-sub\sq\ $(n_{k_l})$ along which the pair $(x,y)$ semi-generates an \im\ $\xi$ on $(K\times K)^\N$. It is elementary to see, that along every such sub-sub\sq\ $(n_{k_l})$ all the limits involved in the definition of correlations, i.e., 
$$
\lim_l \frac1{n_{k_l}}\sum_{i=1}^{n_{k_l}}x_i\overline{y_i}, \ \ \ \ \lim_l \frac1{n_{k_l}}\sum_{i=1}^{n_{k_l}}x_i, \ \ \ \ \lim_l \frac1{n_{k_l}}\sum_{i=1}^{n_{k_l}}\overline{y_i},
$$
exist and equal the respective integrals
$$
\int x_1\overline{y_1}\,d\xi, \ \ \ \int x_1\,d\mu, \ \ \ \int \overline{y_1}\,d\nu,  
$$
where $\mu$ and $\nu$ are the marginal measures of $\xi$ (i.e., $\xi$ is a \emph{joining} of $\mu$ and $\nu$).

Of course, the most convenient situation occurs when the points $x$ and $y$ are generic for some 
\im s. In this case the entire sequences $\frac1n\sum_{i=1}^n x_i$ and $\frac1n\sum_{i=1}^n\overline{y_i}$ (but not $\frac1n\sum_{i=1}^n x_i\overline{y_i}$) converge, which simplifies many arguments. Because we are interested in studying correlation with strictly ergodic \sq s (and every element of a strictly ergodic system is generic), we can immediately assume that our \sq\ $y$ is generic. It is thus reasonable to first look at \sq s $x$ which are also generic (we will extend our results to general elements $x$ in the last but one section). In this manner we are led to the following question:

\begin{que}\label{que2}
When is a \sq\ $x\in K^\N$, generic for a shift-\im, weakly (strongly) correlated to a strictly ergodic \sq?
\end{que}

\medskip
The first (but already rich in consequences) reduction of the problem relies on an observation made long ago by B. Weiss (\cite{weiss}).

\begin{thm}\label{thm1} If $x\in K^\N$ is generic for an \emph{ergodic} measure then it is a ${\bar d}$-limit of a sequence of strictly ergodic points, where $\bar d$ is the Besicovitch distance
$$
\bar d(x,y) = \limsup_n   \frac1n\sum_{i=1}^{n}|x_i - y_i|.
$$
\end{thm}

It is easy to see that if $y$ is a \sq\ strongly or weakly correlated to $x$, then the same holds for 
$y$ and any $x'$ sufficiently close to $x$ in $\bar d$.  This remark, together with the fact that every strictly ergodic point is generic for an ergodic measure immediately imply that 
\begin{thm}\label{thm2}
A \sq\ $x\in K^\N$ is weakly (strongly) correlated to a strictly ergodic point if and only if it is weakly (strongly) correlated to a point generic for an ergodic measure.
\end{thm}

This theorem allows us to formulate question \ref{que2} in a simplified, yet equivalent, version:
\begin{que}\label{que3}
When is a \sq\ $x\in K^\N$, generic for an \im\ $\mu$, strongly (weakly) correlated to a \sq\ generic for an ergodic measure?
\end{que}

Notice that if $\mu$ is ergodic then every point generic for $\mu$ is correlated to a point generic for an ergodic measure (namely to itself). An exception occurs when $\mu$ is a pointmass concentrated at a fixpoint. Such measure is ergodic, yet any of its generic points is uncorrelated to any sequence (even to itself). In either case, our question trivializes if $\mu$ is ergodic. As we shall see, in the nonergodic case, the answer depends exclusively on the properties of the \im, not on the choice of the generic point $x$. More precisely, the answer depends on the \emph{ergodic decomposition} of $\mu$, leading to a discovery of some new features of nonergodic \im s.
\medskip

We conclude this section with a remark concerning the M\"obius function $\mmu$. We can now connect Boshernitzan's question with another celebrated conjecture, the Chowla Conjecture. Leaving aside its precise formulation (see \cite{Cho}), its validity would imply that $\mmu$ were generic for a specific ergodic measure. In view of Theorem \ref{thm2}, this would also imply the negative answer to Question~\ref{que1}: the M\"obius function (as being generic for an ergodic measure not concentrated at a fixpoint) would be strongly correlated to some strictly ergodic \sq.

\section{Correlation of measures}\label{sectwo}
Following the discussion of the preceding section, we introduce the correlation of measures.
\begin{defn}\label{defn3}
The \emph{correlation} between two shift-\im s $\mu$ and $\nu$ on $K^\N$ is the number
$$
\corr(\mu,\nu)=\sup_{\xi}\left|\int x_1\overline{y_1} \,d\xi - \int x_1\,d\mu \cdot \int \overline{y_1}\,d\nu\right|,
$$
where $(x,y)\mapsto x_1$ and $(x,y)\mapsto y_1$ are the ``first symbol value'' functions on the Cartesian square of the shift space $K^\N$, and $\xi$ ranges over all joinings of $\mu$ with $\nu$.
The measures are \emph{uncorrelated} if $\corr(\mu,\nu)=0$. Otherwise we say that the measures are \emph{correlated} and any joining for which the above difference is nonzero will be referred to as a
\emph{correlating joining}.
\end{defn}

We can now formulate a question concerning \im s, completely analogous to Question \ref{que2} posed for \sq s generic for \im s. We will say that a measure is \emph{strictly ergodic} if its \tl\ support (viewed as a subshift) is strictly ergodic.

\begin{que}\label{que4}
When is a shift-\im\ $\mu$ supported by $K^\N$, uncorrelated to any strictly ergodic measure?
\end{que}

In order to simplify this question we first prove an analog of Theorem~\ref{thm2}, which
allows to drop the adjective ``strictly'' from the formulation.

\begin{thm}\label{thm2a} 
A shift-\im\ $\mu$ on $K^\N$ is correlated to a strictly ergodic measure if and only if it is correlated to an ergodic measure.
\end{thm}

\begin{proof} Since every strictly ergodic measure is ergodic, one implication is trivial.
Suppose $\mu$ is correlated to an ergodic measure $\nu$. Let $\xi$ be a correlating joining of $\mu$ and $\nu$. Let $(x,y)$ be generic for $\xi$ (such a pair exists in the full shift over $K\times K$). Then $x$ is generic for $\mu$, $y$ is generic for $\nu$, and these points are strongly correlated.
Indeed all the limits of averages in the definition of strong correlation of the \sq s $x$ and $y$ are equal to the respective integrals in the definition of correlation of the measures $\mu$ and $\nu$ with help of the joining $\xi$. Now, by Theorem \ref{thm1}, $y$ can be replaced by a strictly ergodic element $y'$ (which is generic for a strictly ergodic measure $\nu'$), so that $x$ and $y'$ are strongly correlated. Along some sub\sq\ the pair $(x,y')$ semi-generates an \im\ $\xi'$ on $(K\times K)^\N$, and since $x$ and $y'$ are generic for $\mu$ and $\nu'$, the marginals of $\xi'$ are $\mu$ and $\nu'$, i.e., $\xi'$ is a joining of $\mu$ and $\nu'$. Now we need to reverse the preceding argument: all the integrals in the definition of correlation of $\mu$ and $\nu'$ with help of the joining $\xi'$ are equal to the respective limits (along a subsequence) in the definition of correlation between $x$ and $y'$. Since these points are generic and strongly correlated, the limits indicate correlation regardless of the subsequence. So, $\xi'$ is a correlating joining of $\mu$ and $\nu'$.
\end{proof}

Question \ref{que4} takes now a simplified, equivalent form:

\begin{que}\label{que7}
When is a shift-\im\ $\mu$ supported by $K^\N$, uncorrelated to any ergodic measure?
\end{que}

Uncorrelation is a ``weak form'' of \emph{disjointness} (in the sense of Furstenberg), which is the condition that all expressions $\int f(x)g(y) \,d\xi - \int f(x)\,d\mu\cdot\int g(y)\,d\nu$ equal zero, when evaluated for \emph{all} bounded measurable functions $f$ and $g$ of one variable. If $\mu$ and $\nu$ are disjoint then they are obviously uncorrelated. This raises two further natural questions:
\begin{que}\label{que5}
Are there \im s disjoint from all ergodic measures?
\end{que}
\begin{que}\label{que6}
Does the reversed implication hold: does uncorrelation to any strictly ergodic measure imply disjointness from all (strictly) ergodic measures?
\end{que}

Question \ref{que5} is answered positively by a relatively simple example (provided in the last section). 

As far as Question \ref{que6} is concerned, it is rather hard to expect that uncorrelation for just one specific function (the ``first symbol value'') should imply disjointness. On the other hand, in the case of symbolic systems, this particular function corresponds in fact to a generating partition, and in many proofs in ergodic theory it suffices to consider a generating partition. It turns out that the answer to this question is negative (which makes uncorrelation to any ergodic measure a new property). However, the appropriate example, that we provide in the last section, is far from trivial; in order to verify the desired property we needed to establish  new criteria for ergodicity of certain types of cocycle extensions. Our example has inspired us to introduce the notion of conditional disjointness. This 
idea and its applicability to studying uncorrelation to any ergodic measure are presented in the Appendix at the end of the paper.

Since uncorrelation is essentially weaker than disjointness, another question arises:
\begin{que}\label{que8}
Is the property of being uncorrelated to any ergodic measure an isomorphism invariant?
\end{que}

Again, the answer turns out negative. This property is not even invariant under topological conjugacy. That is, after transforming the shift space via an injective sliding block code (in this manner the shift space is modeled inside another symbolic space), a measure uncorrelated to any ergodic measures may lose this property. It is so, because the ``first symbol value'' function $x_1$ may dramatically change in the sense of information content. This is illustrated in another example provided in the final section. So, the property must be regarded as one of the shift-\im\ in a particular symbolic representation. 

\smallskip

We continue with further criteria for correlation with an ergodic measure. Note that the following one \emph{is} an isomorphism invariant:

\begin{thm}\label{thm4}
If the ergodic decomposition of $\mu$ has an atom, which is not the Dirac measure at a fixpoint of the shift transformation, then $\mu$ is correlated to an ergodic measure.
\end{thm}

\begin{proof} Let $\nu$ be the ergodic measure which is the atom of the ergodic decomposition of $\mu$. That is, 
$$
\mu = p\nu + (1-p)\mu',
$$ where $p\in (0,1]$ and $\mu'$ is some \im. Then $\mu$ admits the following joining with $\nu$: 
$$
\xi = p\nu_\Delta + (1-p)(\mu'\times\nu), 
$$
where $\nu_\Delta$ is the \emph{identity joining} of $\nu$ with itself, supported by the diagonal, and
$\mu'\times\nu$ denotes the product measure. It is now elementary to verify the correlation between $\mu$ and $\nu$ using $\xi$:
\begin{multline*}
\int x_1\overline{y_1}\,d\xi - \int x_1\,d\mu\int\overline{y_1}\,d\nu = \\
p\int x_1\overline{y_1}\,d\nu_\Delta + (1-p)\int x_1\overline{y_1}\,d(\mu'\times\nu) -  
\left(p\int x_1\,d\nu + (1-p)\int x_1\,d\mu'\right) \int \overline{y_1}\,d\nu =\\
p\int |x_1|^2\,d\nu +(1-p)\int x_1\,d\mu'\int\overline{y_1}\,d\nu - p\left|\int x_1\,d\nu\right|^2  - (1-p)\int x_1\,d\mu'\int\overline{y_1}\,d\nu =\\
p\left(\int |x_1|^2\,d\nu - \left|\int x_1\,d\nu\right|^2 \right)\ge 0,
\end{multline*}
with equality holding only when $x_1$ is constant $\nu$-almost surely, that is when $\nu$ is concentrated 
at a fixpoint of the shift transformation.
\end{proof}

The theorem allows to determine, in particular, that an \im\ being a convex combination of finitely 
many ergodic components, at least one of which is not a pointmass, is correlated to some (strictly) ergodic measure.
\medskip

The main result of this section shows that Questions \ref{que2} and \ref{que7} are in fact equivalent. Note that correlation of measures has no weak or strong form.

\begin{thm}\label{thm5}
Suppose that $x\in K^\N$ is generic for an \im\ $\mu$. The following conditions are equivalent:
\begin{enumerate}
	\item $x$ is weakly correlated to a point generic for an ergodic measure,
	\item $x$ is strongly correlated to a point generic for an ergodic measure, 
	\item $\mu$ is correlated to an ergodic measure.
\end{enumerate}
\end{thm}

\begin{proof}
Suppose that $x$, which is generic for $\mu$, is weakly correlated to some $y$ generic for an ergodic measure $\nu$. This means that there exists a subsequence $(n_k)$ such that the limit $\lim_k\corr(x_{[1,n_k]},y_{[1,n_k]})$ exists and is different from zero. There is a sub-sub\sq\
$(n_{k_l})$ along which the pair $(x,y)$ semi-generates an \im\ $\xi$ on $(K\times K)^\N$. Exactly as in 
the proof of Theorem \ref{thm2a} (the argument involving $y'$, $\nu'$ and $\xi'$), $\xi$ is a correlating joining of $\mu$ and $\nu$.

Now, suppose that a point $x$ is generic for an \im\ $\mu$ which is correlated to an ergodic measure $\nu$. 
Let $\xi$ be a correlating joining of $\mu$ and $\nu$. We need to refer to the following fact, whose proof occupies the next section.

\begin{thm}\label{lem1}
Let $\xi$ be a joining of two shift-\im s $\mu$ and $\nu$ supported by $K^\N$, and let $x\in K^\N$ be a point generic for $\mu$. Then there exists a point $y\in K^\N$ such that the pair $(x,y)$ is generic for $\xi$ (in particular, $y$ is generic for $\nu$).
\end{thm}

It is obvious that by applying the above theorem to our situation we obtain a point $y$ generic for the ergodic measure $\nu$ and such that $x$ and $y$ are strongly correlated. This ends the proof of Theorem \ref{thm5}.
\end{proof}

\section{Generic points for joinings}

This section is devoted to proving Theorem \ref{lem1}. To avoid ambiguity, we will distinguish between \emph{free blocks} of length $m$, i.e., the elements of $K^m$ (there are precisely $\#K^m$ free blocks) and \emph{blocks} of length $m$, i.e., subblocks of length $m$ of a longer block $B$ or of a symbolic element $x$. There are precisely $|B|-m+1$ blocks of length $m$ in $B$. Each block is an \emph{occurrence} of a free block. Every free block $B\in K^n$ determines, for every $m\le n$, what we call the \emph{empirical measure} $\mu_B$ on the finite space $K^m$ of free blocks of length $m$, by the formula
$$
\mu_B(D) = \tfrac1n\#\{i\in[1,n-m+1]: B_{[i,i+m-1]}=D\}.
$$
(for the ease of computations, we divide by $n$ rather than by the more commonly used denominator $n-m+1$, as a result we obtain a sub-probabilistic vector).

\begin{defn} For (probabilistic or sub-probabilistic) measures on $K^{m}$ we define the distance
$$
d^{(m)}(\mu, \nu) = \sum_{D\in K^m}|\mu(D) - \nu(D)|.
$$
A block $B$ is said to be $(m,\epsilon)$-generic for an invariant (probability) measure $\mu$ on~$K^\N$ 
if $d^{(m)}(\mu_B,\mu)<\epsilon$. 
\end{defn}
The sub-probabilistic normalization of the empirical measures causes that any block $(m,\epsilon)$-generic for a probability measure has length at least $\frac m\epsilon$. 

Notice that a point $x$ is generic for a measure $\mu$ if and only if, for every positive integer $m$ and $\epsilon>0$, the blocks $x_{[1,n]}$ are eventually (i.e., for large $n$) $(m,\epsilon)$-generic for $\mu$.
\smallskip

We will need two technical lemmas.

\begin{lem}\label{gener}
Assume that $B\in K^n$ and $B_{[1,l]}$ are both $(m,\epsilon)$-generic for $\mu$, where  $l<n(1-\sqrt\epsilon)$. Then $B_{[l+1,n]}$ is $(m,3\sqrt\epsilon)$-generic for $\mu$.
\end{lem}

\begin{proof}
We have
$$
\sum_{D\in K^m}\left|\#\{i\in[1,n-m+1]:B_{[i,i+m-1]}=D\}-\mu(D)\right| = nd^{(m)}(\mu_B,\mu) < n\epsilon 
$$
and the same for $l$ in place of $n$. Thus, subtracting sidewise, we obtain
\begin{multline*}
\sum_{D\in K^m}\left|\#\{i\in[l-m+2,n-m+1]:B_{[i,i+m-1]}=D\}-(n-l)\mu(D)\right|<\\(n+l)\epsilon,
\end{multline*}
The above calculation misses $m-1$ blocks of length $m$ appearing on the left end of $B_{[l+1,n]}$. Adding this number to the right hand side and dividing both sides by $n-l$ we obtain
$$
\sum_{D\in K^m}\left|\mu_{B_{[l+1,n]}}(D)-\mu(D)\right|<\\\frac{n+l}{n-l}\epsilon + \frac{m-1}{n-l}\le \frac {3\epsilon n}{n-l}
$$
(we have also used the inequality $m<\epsilon n$, which holds since $B$ is $(m,\epsilon)$-generic).
Applying the inequality $n-l>n\sqrt\epsilon$ we finish the proof.
\end{proof}

The easy proofs of the following facts are left to the reader. All blocks addressed below are longer than $m$, the empirical measures are defined on $K^m$, and the distance between them is $d^{(m)}$.

\begin{lem}\label{generq} {\color{white}.}
\begin{enumerate}
	\item Let $B\in K^n$. Then the empirical measures determined by $B$ and $B_{[l+1,n-k]}$ 
	are at most $2\frac{l+k}n$ apart.	
	\item Suppose the empirical measures determined by two blocks $B$ and $C$ of the same length are less than $\epsilon$ apart, and the same holds for a pair $B',C'$. Then the empirical measures determined by the concatenations $BC$ and $B'C'$ are less than $2\epsilon$ apart, assuming that the joint length of the concatenation exceeds $\frac m\epsilon$.
\end{enumerate}	
\end{lem}

We pass to the main proof of this section.

\begin{proof}[Proof of Theorem \ref{lem1}]

Let us start by taking a pair $(x',y')$, generic for the joining $\xi$ (such a pair exists in $(K\times K)^\N$). Clearly, the points $x'$ and $y'$ are generic for the projections $\mu$ and $\nu$, respectively. We will think of $(x',y')$ as of a sequence consisting of two rows, $x'$ in the first row, $y'$ in the second.

Choose a decreasing to zero \sq\ of positive numbers $\epsilon_m$ ($m\ge 1$). By the Rokhlin Lemma applied to the system with the measure $\mu$ we have: for each $m$ there exists a Rokhlin tower of height $m$, with remainder of measure less than $\epsilon_m$. By a standard argument, the bases of the towers (and hence the remainders) can be chosen closed-and-open (we will say \emph{clopen}). By lifting to the product space we also have Rokhlin towers for the measure $\xi$.

Throughout the following three paragraphs we fix one value of the parameter $m$. For each point in the product system (i.e., \sq\ in $(K\times K)^\N$), we mark the times of the visits in the base of the Rokhlin tower by an additional symbol (star) and the visits in the remainder of the tower by another additional symbol, say, a dagger. Because we chose the bases clopen, adding the stars and daggers is a topological conjugacy of $(K\times K)^\N$ with a subsystem of $(\K\times\K)^\N$, where $\K = K\times\{\star,\dagger,\emptyset\}$. In particular, generic points are preserved. We can think of $\xi$ as a measure supported by $(\K\times\K)^\N$ (and $\mu$ as a measure supported by $\K^\N$), we will call this the \emph{marked representation}. Every point generic for $\mu$ (including $x$ and $x'$) has, in this representation, the structure of a concatenation of blocks of the length $m$ starting with the star (we will call these blocks \emph{$m$-blocks}), separated by strings, of various lengths, of symbols marked by the dagger (the \sq\ may start with a prefix of length at most $m-1$, without any markers). The daggers occur with density less than $\epsilon_m$, the stars have density between  $\frac{1-\epsilon_m}m$ and $\frac1m$, every free $m$-block occurs with the density equal to its measure. We let $l_m$ be such that for every $n\ge l_m$ the block $(x',y')_{[1,n]}$ is $(m,\epsilon_m)$-generic for $\xi$ (in particular, the block $x'_{[1,n]}$ is $(m,\epsilon_m)$-generic for $\mu$) and also the block $x_{[1,n]}$ is $(m,\epsilon_m)$-generic for $\mu$. We can inductively arrange that $l_{m+1}$ is larger than $\frac{l_m}{\epsilon_m}$. This proportion and Lemma \ref{generq} (1) imply that for $n\ge l_{m+1}$ the blocks $(x',y')_{[l_m+1,n]}$ and $x_{[l_m+1,n]}$ are $(m,3\epsilon_m)$-generic for $\xi$ and $\mu$, respectively.

\smallskip
We will now describe the inductive $m$th step of our construction, in which we define the second row $y$ in the pair $(x,y)$ between the positions $l_m+1$ and $l_{m+1}$ (we let $y_{[1,l_1]}$ be defined arbitrarily). 

Given a free $m$-block $A$ over the alphabet $K$, we locate all its occurrences in both $x_{[l_m+1,l_{m+1}]}$ and $x'_{[l_m+1,l_{m+1}]}$ (sitting completely inside). Next we define $y$ at the positions corresponding to the occurrences of $A$ in $x$ proceeding from left to right and rewriting consecutive symbols of $y'$ from the positions corresponding to the occurrences of $A$ in $x'_{[l_m+1,l_{m+1}]}$, maintaining the order. We continue until we exhaust the available positions in either $x_{[l_m+1,l_{m+1}]}$ or $x'_{[l_m+1,l_{m+1}]}$ (or both). We perform this procedure separately, for every free $m$-block $A$. Note that since the $m$-blocks always occur separately (without overlapping), there will be no collision at any position of $y$. As a result, nearly all two-row $m$-blocks occurring in $(x',y')_{[l_m+1,l_{m+1}]}$ will be copied in $(x,y)_{[l_m+1,l_{m+1}]}$; there is a 1-1 correspondence between nearly all $m$-blocks here and here. The correspondence is defined except on a small percentage of the $m$-blocks here and here, mainly occurring near the right end (for each free $m$-block $A$ the number of occurrences in $x'_{[l_m+1,l_{m+1}]}$ may be slightly larger or slightly smaller than in $x_{[l_m+1,l_{m+1}]}$ -- in either case the \emph{excessive} $m$-blocks on one side have to be excluded from the 1-1 correspondence). Precisely, since both first-row blocks are $(m,3\epsilon_m)$-generic for $\mu$, this percentage does not exceed $6\epsilon_m$. In the end, there may remain some unfilled positions in $y_{[l_m+1,l_{m+1}]}$; we fill them arbitrarily. This concludes the definition of~$y$. 
\smallskip

In order to verify that the above defined pair $(x,y)$ is generic for $\xi$ we will compare, for every positive integer $m_0$, the empirical measures on $(K\times K)^{m_0}$ determined by the initial blocks $(x,y)_{[1,n]}$ and $(x',y')_{[1,n]}$. Since the latter tend to $\xi$ with increasing $n$, it suffices to prove that the distance between the above two empirical measures tends to zero. So, we fix an $m_0$ and $n$ larger than $l_{m_0+1}$. We let $m\ge m_0$ be the largest index such that $n \ge l_{m+1}$. Note that $m$ implicitly depends on $n$ and tends to infinity as $n$ grows, so that all expressions of order $\epsilon_m$, $\sqrt{\epsilon_m}$ or $\frac1m$ tend to zero with $n$. To avoid notorious repetitions of the phrase ``blocks of length $m_0$ over $K\times K$'' (as opposed to other blocks referred to in the argument) we will call them shortly the \emph{words}. From now on \emph{empirical measures} are always defined on so understood free words.
\smallskip

We begin by proving that the empirical measures determined by $(x,y)_{[l_m+1,l_{m+1}]}$ and $(x',y')_{[l_m+1,l_{m+1}]}$ are close. We classify the words contained in $(x,y)_{[l_m+1,l_{m+1}]}$ in four groups:
\begin{enumerate}	
	\item[(I)] words not inside any $m$-block; these can be recognized by either having a star \emph{not} at position 1, or a dagger anywhere (we mean the stars and daggers created in step $m$), 
	\item[(II)] words contained inside any of the two possible $m$-blocks appearing at either end of and only partly contained in $(x,y)_{[l_m+1,l_{m+1}]}$,
	\item[(III)] words contained within an $m$-block sitting inside $(x,y)_{[l_m+1,l_{m+1}]}$, but being	the ``excessive'' $m$-block (excluded from the aforementioned 1-1 correspondence),
	\item[(IV)] words contained within an $m$-block sitting inside $(x,y)_{[l_m+1,l_{m+1}]}$ and included in the 1-1 correspondence. 
\end{enumerate}

By the $(m,3\epsilon_m)$-genericity and thus also $(1,3\epsilon_m)$-genericity of $x_{[l_m+1,l_{m+1}]}$ for $\mu$, the first group constitutes a small percentage of all considered words (smaller than $m_0(3\epsilon_m+\frac1m+\epsilon_m)$). The same genericity implies small proportion between $m$ and the length of $x_{[l_m+1,l_{m+1}]}$, hence smallness of the second group. We already know that among all $m$-blocks contained in $(x,y)_{[l_m+1,l_{m+1}]}$ only
a small percentage (at most $6\epsilon_m$) is excluded from the 1-1 correspondence. This implies that the third group is small, as well. We have shown that the last group dominates (constitutes a percentage converging to 1 with growing $n$) of all words in $(x,y)_{[l_m+1,l_{m+1}]}$. The 1-1 correspondence between 
the $m$-blocks induces, in an obvious way, a 1-1 correspondence between words in group (IV) and words in the
analogous group (IV) regarded for $(x',y')_{[l_m+1,l_{m+1}]}$ (which, by a symmetric argument, dominates among all words in $(x',y')_{[l_m+1,l_{m+1}]}$). This clearly implies that the empirical measures determined by $(x',y')_{[l_m+1,l_{m+1}]}$ and $(x,y)_{[l_m+1,l_{m+1}]}$ are close, as desired.
\smallskip

Now consider two cases: 
\begin{enumerate}
	\item[(a)] $n-l_{m+1}\le n\sqrt{\epsilon_m}$, \  and 
	\item[(b)] the opposite. 
\end{enumerate}

In case (a), $(x,y)_{[l_m+1,l_{m+1}]}$ is obtained from $(x,y)_{[1,n]}$ by truncating at most
$n(\epsilon_m+\sqrt{\epsilon_m})$ terms at the ends, and the same holds for $(x',y')$. Thus, Lemma~\ref{generq}~(1) (applied twice) and the closeness of the empirical measures of $(x,y)_{[l_m+1,l_{m+1}]}$ and $(x',y')_{[l_m+1,l_{m+1}]}$ imply closeness of the empirical measures of $(x,y)_{[1,n]}$ and $(x',y')_{[1,n]}$, which concludes this case.
\smallskip

In case (b), we will argue that also the empirical measures determined by the blocks $(x,y)_{[l_{m+1}+1,n]}$ and $(x',y')_{[l_{m+1}+1,n]}$ are close. Observe that the closeness of the empirical measures determined by
$(x,y)_{[l_m+1,l_{m+1}]}$ and $(x',y')_{[l_m+1,l_{m+1}]}$ was deduced using exclusively the fact
that the blocks $x_{[l_m+1,l_{m+1}]}$ and $x'_{[l_m+1,l_{m+1}]}$ were $(m,3\epsilon_m)$-generic for $\mu$.
Now we are in a very similar situation: since $l_{m+1}<n(1-\sqrt{\epsilon_m})$, we can use Lemma~\ref{gener} to conclude that the blocks $x_{[l_{m+1}+1,n]}$ and $x'_{[l_{m+1}+1,n]}$ are both $(m+1,3\sqrt{\epsilon_{m+1}})$-generic for $\mu$.
Because the algorithm of defining $y_{[l_{m+1}+1,l_{m+2}]}$ proceeds from left to right, we can stop it when we reach the coordinate $n$ and we will have the 1-1 correspondence already defined between majority of $(m\!+\!1)$-blocks in $(x,y)_{[l_{m+1}+1,n]}$ and $(x',y')_{[l_{m+1}+1,n]}$. From here the closeness of the empirical measures determined by $(x,y)_{[l_{m+1}+1,n]}$ and $(x',y')_{[l_{m+1}+1,n]}$ (with $\sqrt{\epsilon_{m+1}}$ replacing $\epsilon_m$ in the estimates) follows as in the preceding argument. 

Now Lemma \ref{generq}~(2) yields closeness of the empirical measures of $(x,y)_{[l_m+1,n]}$ and
$(x',y')_{[l_m+1,n]}$ (viewed as appropriate concatenations) and one more application of (1) extends this to $(x,y)_{[1,n]}$ and $(x',y')_{[1,n]}$. The proof in case (b) is complete.
\end{proof}

\section{General \sq s}
A general \sq\ over the alphabet $K$ is typically not generic for any \im\ but it is semi-generic for a 
range of \im s. As we will show, in such case we can decide about its uncorrelation versus weak correlation to a (strictly) ergodic \sq\ by examining all the measures semi-generated by $x$. We cannot expect strong correlation results in this case. The main theorem of this section requires a lemma similar to Theorem \ref{lem1}.

\begin{lem}\label{semgengen}
Let $x\in K^\N$ be semi-generic, along a subsequence $(n_k)$, for an \im\ $\mu$, and let $\xi$ be a
joining of $\mu$ with a strictly ergodic measure $\nu$. Then there exists a point $y$ generic for 
$\nu$ and such that the pair $(x,y)$ semi-generates $\xi$ along a sub-sub\sq\ of $(n_k)$. 
\end{lem}

\begin{proof} We will only outline the proof, focusing on the details which are different than in the proof of Theorem \ref{lem1}.

We pick a pair $(x',y')$ generic for $\xi$. Since $\nu$ is strictly ergodic, we can assume that $y'$
is strictly ergodic, which implies that it is \emph{uniformly generic}, i.e., for every integer $m$ 
and $\epsilon>0$ there exists an $n$ such every block of length at least $n$, occurring in $y'$, is $(m,\epsilon)$-generic for $\nu$. 

As in the other proof, we fix a decreasing to zero \sq\ $\epsilon_m$ and, for each $m$, we create the
marked representation with the $m$-blocks. We select the lengths $l_m$ from the sub\sq\ $(n_k)$ so that 
the blocks $x_{[1,l_m]}$ and $(x',y')_{[1,l_m]}$ are $(m,\epsilon_m)$-generic for $\mu$ and $\xi$ respectively. As before, we arrange that $l_{m+1}$ is larger than $\frac{l_m}{\epsilon_m}$, but we do 
not care about genericity of initial blocks of other lengths $n$. The inductive step of defining $y$ is identical as before. The verification that $(x,y)$ semi-generates $\xi$ along $(l_m)$ (which is a sub-sub\sq\ of $(n_k)$) is simplified; since we do not care about other lengths $n$ we do not need to consider the two cases (a) and (b), or invoke Lemma \ref{gener}. 

We need, however, an additional argument to prove that $y$ is generic for $\nu$ (it is obvious that $y$ semi-generates $\nu$ along $(l_m)$). But this fact follows easily from the assumption that $y'$ is uniformly generic for $\nu$ and that $y$ is built as a concatenation (with insertions of density zero) of longer and longer blocks occurring in $y'$.
\end{proof}

\begin{thm}\label{genseq}
A \sq\ $x\in K^\N$ is weakly correlated to a \sq\ generic for an ergodic measure if and only if $x$ 
is semi-generic for at least one measure correlated to an ergodic measure.
\end{thm}
Rephrasing the theorem, $x$ is uncorrelated to any \sq\ generic for an ergodic (equivalently, strictly ergodic) measure if and only if all of the \im s semi-generated by $x$ have the property of being uncorrelated to any ergodic measure.

\begin{proof} If $x$ is weakly correlated to a point $y$ generic for an ergodic measure $\nu$ then there is a sub\sq\ $(n_k)$ such that the limit $\lim_k\corr(x_{[1,n_k]},y_{[1,n_k]})$ exists and is different from zero. Choosing a sub-sub\sq\ we can assume that the pair $(x,y)$ semi-generates an \im\ $\xi$ on the product space. Clearly, the second marginal of $\xi$ is $\nu$, while the first marginal is an \im\ $\mu$ for which $x$ is semi-generic and which is correlated with $\nu$ (via the joining $\xi$).

Now suppose that $x$ semi-generates an \im\ $\mu$ which is correlated to an ergodic measure $\nu$ via a joining $\xi$. By Theorem \ref{thm2a}, we can assume that $\nu$ is strictly ergodic. Now, Lemma \ref{semgengen} allows to couple $x$ with a point $y$ generic for $\nu$ so that
$(x,y)$ semi-generates $\xi$. It is clear that $x$ and $y$ are weakly correlated.
\end{proof}
 
\section{Examples}
 
\begin{exam}
This is an example of an \im\ on $K^\N$, disjoint from (hence uncorrelated to) any ergodic measure.
Consider the mapping $T(t,s)=(t,s+t)$ on the two-dimensional torus, equipped with the product Lebesgue measure $dt\times ds$. Ignoring a set of measure zero, the space decomposes to invariant circles, where on each circle we have a different irrational (hence ergodic) rotation. Suppose that an ergodic measure $\nu$ is not disjoint with $dt\times ds$; it follows that the set $A$ of such parameters $x$ that $\nu$ is not disjoint with the irrational rotation by the angle $x$, has positive Lebesgue measure in the base. The fact that $\nu$ is not disjoint with an irrational rotation is equivalent to $\nu$ having an eigenvalue rationally dependent with the rotation angle. Since the ergodic measure $\nu$ possesses at most countably many different eigenvalues, it follows that the set $A$ is at most countable, hence of measure zero, implying disjointness. Now, we take any (measurable) finite partition of the 2-torus which partitions nontrivially every ergodic circle, and label its members by elements of a finite subset $K$ of the unit disc. This produces a symbolic factor $\mu$ of $dt\times ds$ on $K^\N$ disjoint from all ergodic measures (as a factor of such, nontrivial on each ergodic component).
\end{exam}

\begin{exam}\label{example}
There exists a shift-\im\ on $K^\N$ uncorrelated to any ergodic measure, yet not disjoint with an irrational rotation (in fact being an extension of such a rotation).

Consider the direct product of the identity on the circle $S_1=[0,1)$ ($0=1$) with an irrational 
rotation (by $\alpha$) on $S_1$ (represented as $(\cdot +\alpha)$ mod 1). 
Let $A\subset S_1\times S_1$ be the triangle $A=\{(t,s): t\in[0,1),\ 0\le s< t\}$ and let 
$$
\varphi(t,s) = e^{\pi i\mathbf 1_A(t,s)}
$$ 
(i.e., $\varphi$ equals $-1$ on $A$ and $1$ otherwise). Consider the $\Z_2$-extension ($\Z_2$ written multiplicatively, as $\{-1,1\}$) corresponding to the cocycle $\varphi$:
$$
T_\varphi(t,s,\kappa) = (t,s+\alpha, \kappa\cdot\varphi(t,s)), 
$$
where $(t,s,\kappa)\in X=S^2_1\times\Z_2$.
This mapping preserves the measure $\mu$ which is the product of the Lebesgue measure on $S_1^2$ 
and the Haar measure on $\Z_2$: $d\mu=dt\times ds\times d\kappa$, and the corresponding measure-preserving system $(X,T_\varphi,\mu)$ is clearly an extension of the irrational rotation by $\alpha$, which appears on the second coordinate. Thus, this measure-preserving system is \emph{not disjoint} from the ergodic system represented by the rotation. It is easy to see that the above $\Z_2$-extension has a symbolic representation over two symbols $\{-1,1\}$ obtained by replacing each point by its $\Z_2$-forward itinerary
$$
(t,s,\kappa) \mapsto (x_n)_{n\ge 1},
$$
where 
$$
x_n = \kappa \prod_{i=0}^{n-2} \varphi(t,s+i\alpha),
$$
where the product of zero terms (occurring for $n=1$) equals, by convention, 1. In this representation, $\mu$ becomes a shift-invariant measure supported by $K^\N$, where $K = \{-1,1\}$ is a finite subset of the unit disc, 
and hence fits in the framework of Section \ref{sectwo}. Notice that the ``first symbol value'' 
function in this symbolic representation coincides simply with $\kappa$ in the skew product representation.
\smallskip

We will argue that so defined $\mu$ is uncorrelated to any ergodic measure supported by any complex-valued subshift. In what follows, we will switch freely between the symbolic and skew product representation of $\mu$, depending on our needs.

Clearly, $\mu$ is not ergodic, its ergodic components are measures $\mu_t$ supported by the cocycle extensions of the circle rotation with fixed parameter $t$, and the ``section cocycle'' $\varphi_t$ (equal $-1$ on the arc $[0,t)$ and $1$ otherwise). Of course, not all such cocycle extensions are ergodic (meaning that the measure $ds\times d\kappa$ need not be ergodic), but as we will explain in the next section, the set of parameters $t$ for which this happens has measure zero (see the statement \eqref{state}), so in the ergodic decomposition of $\mu$ such parameters may be ignored.

Let $\nu$ be an ergodic measure on a complex-valued subshift. Since $\mu$ has zero mean (the integral of $x_1$ is zero), in order to prove that $\nu$ is uncorrelated to $\mu$ we need to show that $\int y_1 x_1 \, d\xi = 0$ (the complex conjugate can be skipped because $x_1$ is real) for any joining $\xi$ of $\nu$ and $\mu$. Let $\xi = \int \xi_t\,dt$ be the disintegration of $\xi$ with respect to $dt$, so that $\xi_t$ is a joining of $\nu$ with $\mu_t$. 

We have $\int y_1 x_1 \, d\xi = \int (\int y_1 x_1 \, d\xi_t)\,dt$ hence it suffices to show that the inner integral vanishes for almost every $t$. We will do it by proving the following claim:

\begin{itemize}
	\item There exists a set $E\subset S_1^2$ of full product Lebesgue measure, such that whenever a \sq\ $(t_j)_{j\ge 1}$ satisfies, for every $j\neq j'$, the condition $(t_j,t_{j'})\in E$, then the sequence $\int y_1 x_1 \, d\xi_{t_j}$ tends to zero with $j$.
\end{itemize}

At first we argue why is this claim sufficient. Suppose that $\int y_1 x_1 \, d\xi_t\neq 0$ on a positive measure set of parameters $t$. Then, for some $\epsilon>0$ the inequality $|\int y_1 x_1 \, d\xi_t|\ge \epsilon$ also holds on a positive measure set $F$ of parameters $t$. There exists $t_1\in F$ such that the $t_1$-section of $E$ has full measure. Then there exists $t_2\in F$ belonging also to the aforementioned $t_1$-section of $E$ and such that the $t_2$-section of $E$ has full measure. Inductively, once $t_1,\dots,t_j$ are selected, we pick $t_{j+1}\in F$ belonging to all the $t_{j'}$-sections of $E$ for ${j'}\le j$ (and having a full measure section of $E$ itself). Such \sq\ $(t_j)$ satisfies the condition that all distinct pairs are in $E$, while the corresponding \sq\ of integrals does not tend to zero, so the claim does not hold.

In order to prove the claim, assume temporarily that the required set $E$ exists and fix a \sq\ $(t_j)$ as described above (with all distinct pairs in $E$). From now on, we will abbreviate the indexes $t_j$ by $j$ (and write $\mu_j$, $\xi_j$, $\varphi_j$ instead of $\mu_{t_j}$, $\xi_{t_j}$, $\varphi_{t_j}$, respectively). Consider a countable joining $\zeta$ of the measures $\nu,\mu_1,\mu_2,\mu_3,\dots,$ such that for every $j\ge 1$ the projection of $\zeta$ jointly on the zero'th and $j$th coordinate equals $\xi_j$ (there exists such a joining: all joinings $\xi_j$ have the common factor $\nu$, hence we can take $\zeta$ to be their relatively independent joining over the common factor).

On the product space supporting $\zeta$ we have the function $y_1$ (depending only on the zero'th coordinate) and the functions $x_1^{(1)}, x_1^{(2)}, x_1^{(3)},\dots$, where $x_1^{(j)}$ depends only on the $j$th coordinate and represents the first symbol value on the support of $\mu_j$. All these functions are measurable and bounded, so they belong to $L^2(\zeta)$. The integrals $\int y_1 x_1 \, d\xi_j$ can be written as $\int y_1 x_1^{(j)}\, d\zeta$, i.e., they become the inner products $\langle y_1, x_1^{(j)}\rangle$ in $L^2(\zeta)$. If we knew that the functions $x_1^{(j)}$ were pairwise orthogonal (they are obviously normalized), then the above inner products would be the Fourier coefficients of the projection of $y_1$ onto the subspace spanned by the functions $x_1^{(j)}$ and thus they would form a \sq\ belonging to $\ell^2$, in particular they would converge to zero, as needed.

For orthogonality of $x_1^{(j)}$ and $x_1^{(j')}$ (for $j\neq j'$) we need to check that the integral $\int x_1^{(j)} x_1^{(j')}\, d\zeta$ equals zero. This integral equals $\int x_1^{(j)} x_1^{(j')}\, d\zeta_{j,j'}$, where $\zeta_{j,j'}$ is the projection of $\zeta$ onto jointly the $j$th and $j'$th coordinates. Clearly $\zeta_{j,j'}$ is a joining of $\mu_j$ and $\mu_{j'}$, so, in fact, it suffices to show that $\mu_j$ and $\mu_{j'}$ are uncorrelated. Since both measures are ergodic, every their joining decomposes to ergodic joinings, thus it suffices to examine their ergodic joinings (denoted henceforth by $\theta$) only. 

Now we must go back to the original skew product representation and study possible ergodic joinings $\theta$ of $\mu_t$ and $\mu_{t'}$ (where $t, t'$ abbreviate $t_j$ and $t_{j'}$, respectively). Let $(x,x')$ be a pair generic for $\theta$. This pair is obtained in the following manner: we choose two points, $s_0$ and $s_0'$ on the circle, and two initial values, $\kappa_0$ and $\kappa_0'$, from $\Z_2$, and then $x$ and $x'$ are given (as symbolic \sq s) by the rule
\begin{align*}
x_n &= \kappa_0 \prod_{i=0}^{n-2} \varphi_t(s_0+i\alpha),\\
x'_n &= \kappa'_0 \prod_{i=0}^{n-2} \varphi_{t'}(s'_0+i\alpha)= \kappa_0' \prod_{i=0}^{n-2} \varphi_{t'}(s_0+u+i\alpha),
\end{align*}
where $u=s'_0-s_0$.
By genericity of $(x,x')$, the integral of the product of the first-symbol value functions can be evaluated
as the limit of the averages of the products of the $j$th-symbol values, i.e., we just need to look at the \sq\ $x_nx'_n$ (obtained by coordinatewise multiplication of the above two \sq s). Since $\kappa_0\kappa'_0$ is just another element of $\Z_2$, such \sq\ is obtained as the symbolic representation of the point $(s_0,\kappa_0\kappa'_0)$ in the cocycle extension (of the same rotation by $\alpha$), with the new cocycle $\varphi(s)=\varphi_t(s)\varphi_{t'}(s+u)$. This new cocycle equals $-1$ on the symmetric difference of the intervals $[0,t)$ and $[u,u+t')$ (and 1 on the rest). Suppose that this new cocycle extension is ergodic (with respect to the product measure $ds\times d\kappa$). By a classical theorem of Furstenberg (\cite{F}), this extension is also strictly ergodic, and hence every point (in particular the one we have selected, $(s_0,\kappa_0\kappa_0')$) is generic for
$ds\times d\kappa$. Thus, the limit of the averages we are interested in equals the integral of the
``first symbol value'' function in the symbolic representation of the new cocycle extension, i.e., of the function $(s,\kappa)\mapsto\kappa$. Clearly, the integral of this function with respect to $ds\times d\kappa$ equals zero.

In this manner, we have arrived to the following conclusion: All we need is the existence of a set $E\subset S_1^2$ of full product Lebsgue measure, such that every pair of parameters $(t,t')\in E$, $t\neq t'$ fulfills: 
	\begin{itemize}
	\item for any $u\in S_1$ the cocycle extension corresponding to the cocycle equal to $-1$ on the
	symmetric difference of the intervals $[0,t)$ and $[u,u+t')$ (and 1 on the rest) is ergodic.
\end{itemize}

The following section is devoted to studying ergodicity of four-jump $\Z_2$-extensions. As a corollary of
the criteria which we provide, we will derive that the above set $E$ indeed exists (see Theorem \ref{ergCocy1}). This ends the verification of our example.
\end{exam}

\begin{exam}
This example shows that the property of being uncorrelated to any ergodic measure is not a conjugacy invariant.

In the preceding example, we have a shift-\im\ $\mu$ on $\{-1,1\}^\N$. The code 
$$
\Pi(x)_n = x_n x_{n+1}
$$
has the effect that it turns the cocycle $\varphi$ into a semicocycle, that is it reproduces the 
symbolic system arising from reading the function $\varphi$ along the orbits (without choosing randomly the initial value and without the cumulative multiplication). In this factor, the first symbol value function is equal (up to measure) to the function $\varphi$ on $S_1^2$ with the Lebesgue measure. Now consider the code
$$
\Pi'(x)_n = \frac34 x_n + \frac14 \Pi(x)_n.
$$
This is a conjugacy sending our system to a subshift over four symbols $\{-1,-\frac12,\frac12,1\}$ 
so that the first symbol carries information about both the original first symbol (by just looking at the sign) and the first symbol of the factor by $\Pi$ (by looking at the finer value). We will show that the measure $\Pi'(\mu)$, although it is conjugate to $\mu$ uncorrelated to any ergodic measure, is correlated to an ergodic measure, namely to the rotation by $\alpha$ represented symbolically as a Sturmian system 
given by the semicocycle $\varphi_t$ (the choice of $t$ is in fact arbitrary, but to get the classical Sturmian representation we choose $t=\alpha$). Indeed, for any joining $\xi$ of $\Pi(\mu_\alpha)$ with $\Pi'(\mu)$ we have
\begin{gather*}
\int y_1x_1\,d\xi = \frac34 \int \varphi(\alpha,s')x_1\,d\xi + \frac14\int\varphi(\alpha,s')\varphi(t,s)\,d\xi,
\end{gather*}
where $\xi$ is understood as a joining of the Lebesgue measure on the circle $\{\alpha\}\times S_1$ (here the variable is $s'$) with $\mu$, which in the rightmost integral is replaced by the product Lebesgue measure (here the variables are $t$ and $s$). The central integral equals zero, because it attempts to correlate $\mu$ with an ergodic measure. We can choose $\xi$ so that in the last integral it represents a joining concentrated on the diagonal set $\{s'=s\}$ and here it is the Lebesgue measure $dt\times ds$ (such measure is easily seen to be a joining of $\Pi(\mu_\alpha)$ with $\Pi(\mu)$, so it can be lifted to a joining of $\mu_\alpha$ with $\mu$). Then the integrated function equals $-1$ for pairs $(t,s)$ such that $t\le s <\alpha$ and $\alpha\le s <t$, and $1$ otherwise. As very easy to see, this integral is positive
($\alpha$, being irrational, differs from $0$ or $1$) and so $\xi$ is a correlating joining.
\end{exam}

\section{Ergodicity of four-jump $\Z_2$-cocycles}\label{section6}
In this section we will consider $\Z_2$-exten\-sions $T_\varphi$ of an irrational rotation (by $\alpha$) 
on the circle $S_1$ equipped with the Lebesque measure (denoted by $ds$), determined by a cocycle $\varphi:S_1\to\Z_2$:
$$
T_\varphi(s,\kappa) = (s + \alpha, \kappa\cdot\varphi(s)).
$$
For $n\ge 1$, the $n$th iterate of $T_\varphi$ is determined by the \emph{$n$-step cocycle} 
$$
\varphi^{(n)}(s) = \prod_{j=0}^{n-1}\varphi(s+j\alpha),
$$
namely
$$
T^n_\varphi(s,\kappa) = (s + n\alpha, \kappa\cdot\varphi^{(n)}(s)).
$$
As before, we will say that the cocycle $\varphi$ is ergodic if the product measure $ds\times d\kappa$
is ergodic under $T_\varphi$. One easily proves the following criterion for this $\Z_2$-extension: $\varphi$ is ergodic if and only if it is not a \emph{coboundary}, i.e., there is no measurable solution $\psi$ of the so-called \emph{cohomological equation}, which, in the multiplicative notation, reads:
\begin{equation}\label{eqfunct}
\varphi(s) = \frac{\psi(s + \alpha)}{\psi(s)}. 
\end{equation}
The function $\psi$, if it exists, can be taken with values in $\Z_2$. 
In such case, for each $n\ge 1$ we have
\begin{equation}\label{eqfunct1}
\varphi^{(n)}(s) = \frac{\psi(s + n\alpha)}{\psi(s)}.
\end{equation}
  
The following lemma gives a sufficient condition for ergodicity of a cocycle.

\begin{lem} \label{noSolLem} If there exist a sequence $(n_k)$ of positive integers such that 
\begin{itemize}
	\item $n_k \alpha \to 0$ in $S_1$, and
	\item $\varphi^{(n_k)} \not\to 1$ in $L^1(ds)$,
\end{itemize}
then the cocycle $\varphi$ is ergodic.
\end{lem}
\begin{proof}
If $\varphi$ is not ergodic then \eqref{eqfunct} has a measurable solution $\psi$, and \eqref{eqfunct1} holds. But then $n_k\alpha\to 0$ in $S_1$ implies (via Luzin Theorem) that $\psi(s+n_k\alpha)$
tends to $\psi(s)$ in $L^1(ds)$ which yields that $\varphi^{(n_k)} \to 1$ in $L^1(ds)$.
\end{proof}



\medskip

We will be using the following notation: For $s\in S_1=[0,1)$, we let 
$$
\|s\| = \min (s, 1 - s)
$$ 
(the distance of $s$ to 0 in $S_1$). This quantity satisfies, for $s,s' \in S_1$, the triangle inequality $\|s + s'\| \leq  \|s\| +  \|s'\|$, and $\|q s\| \leq q\|s\|$, for $q \in \N$ (the sums and multiples of elements in $S_1$ are understood modulo 1).
\smallskip

The continued fraction expansion of $\alpha$ will be written as $[0; a_1,..., a_k,...]$ and $(\frac{p_k}{q_k})_{k \ge -1}$ will be the sequence of its convergents. Recall that, for all $k \ge 1$, we have
\begin{align}
&\alpha = \frac{p_k}{q_k} + \frac{\zeta_k}{q_k}, \text{ \ where \ } |\zeta_k| \leq \frac1{q_{k+1}}
\leq \frac1{a_{k+1}} \frac1{q_k}, \ \ \ \|q_ k\alpha\| \leq \frac1{q_{k+1}}, \label{pnqn} \\
&\frac1{2q_k} \le \frac1{q_k + q_{k-1}} \le \|q_{k-1} \alpha\| \leq \|j \alpha\|, \ \forall j: 0<|j| < q_k.
\label{minmajqn}
\end{align}

According to the inequality (\ref{minmajqn}), for $s \in S_1$, the distance between any two elements of the set $\{s - j\alpha : j = 0,\dots, q_k -1\}$ is larger than $\frac1{2q_k}$.
\medskip

For $t, t', u\in S_1$ we will consider the $\Z_2$-extensions given by the cocycles
\begin{itemize}
	\item $\varphi_t = e^{\pi i\mathbf 1_{[0,t)}}$, 
	\item $\varphi_{t,t'}^u(s) = \varphi_t(s)\cdot\varphi_{t'}(s+u)$.
\end{itemize}
The function $\varphi_{t,t'}^u$ equals $-1$ on the symmetric difference of the (positively oriented) arcs $[0,t)$ and $[u,u+t')$ and $1$ otherwise.

\medskip
Ergodicity of the two-jump cocycles $\varphi_t$ has been discussed in a series of papers 
(W. Veech \cite{Ve69}, \cite{Ve75}, K. Merrill \cite{Me85}, M. Guenais and F. Parreau \cite{GuPa06}). 
In the latter paper a complete characterization of ergodicity has been given in terms of Ostrowski's expansion of the parameter $t$ relatively to $\alpha$. These results in particular imply that 
\begin{equation}\label{state}
\text{the set of parameters $t$ for which $\varphi_t$ is ergodic has full Lebesgue measure.}
\end{equation}

Ergodicity of cocycles with 4 jump points has been proved for some families (cf \cite{Me85}, \cite{GuPa06}, \cite{CoPi14}). Nevertheless, for cocycles of the form $\varphi_{t,t'}^u$, it seems that no complete characterization of the ergodic case has been given.
If $t=t'$, as shown in \cite{Me85} (even when $\alpha$ is of bounded type), there exists an uncountable family of parameters $(t,u)$ such that the cocycle $\varphi_{t,t}^u$ is not ergodic. 

\vskip 3mm
We will prove:
\begin{thm} \label{ergCocy0} If the cocycle $\varphi_{t,t'}^u$ is not ergodic
then the parameters $t, t', u$ satisfy simultaneously the following four convergences
\begin{align*}\label{fourcon}
&\lim_k \min\{\|q_k t\|, \|q_k u\|, \|q_k (t'+u)\|\} = 0, \\
&\lim_k \min\{\|q_k t\|, \|q_k (t-u)\|, \|q_k (t - t' - u)\|\} = 0, \\
&\lim_k \min\{\|q_k u\|, \|q_k (t-u)\|, \|q_k t'\|\} = 0, \\ 
&\lim_k \min\{\|q_k (t'+u)\|, \|q_k(t - t' - u)\|, \|q_k t'\|\} = 0,
\end{align*}
where $(q_k)$ is the \sq\ of denominators in the continued fraction expansion of~$\alpha$.
In particular, we have
\begin{equation}\label{minn0}
\lim_k \min\{\|q_k (t-t')\|, \|q_k (t+t')\|\} = 0. 
\end{equation}
\end{thm}

\begin{rem}\label{nonsep}
The points $0, t, u, t'+u$ are discontinuities of the cocycle $\varphi_{t,t'}^u$ (they may appear in $[0,1)$ in a different order). Of course, it may happen that some of these points coincide (and the cocycle has only two or even no discontinuities). If we denote these points as $x_1,x_2,x_3,x_4$, the four convergences in the assertion of the theorem can be written as one condition:
\begin{equation}\label{condC1}
\text{for every } i \in \{1, 2, 3, 4\} \text{ \ \ we have \ \ } \lim_k \, \min_{i' \neq i} \, \|q_k (x_i - x_{i'})\| = 0. 
\end{equation}
\end{rem}

The proof of the theorem will be provided in a moment. First we derive from it the following, important for us, result:

\begin{thm} \label{ergCocy1} The set $E\subset S_1^2$ of pairs $(t,t')$ such that the cocycles $\varphi_{t,t'}^u$ are ergodic for all $u\in S_1$
 has full product Lebesgue measure $dt\times dt'$.
\end{thm}

\begin{proof} By Theorem \ref{ergCocy0} it suffices to show that \eqref{minn0} fails for Lebesgue-almost 
all pairs $(t,t')$. For a moment let $(q_k)_{k\ge1}$ be any strictly increasing \sq\ of positive integers.
We recall the \emph{equirepartition property}: Lebesgue-almost every pair $(x,y)$ satisfies, for every 
continuous function $f$ on $S_1^2$, the condition
$$
\lim_N \frac1N \sum_{k=0}^{N-1} f(q_k x, q_k y) = \int f \, du\, dv. 
$$
Indeed, by approximation of $f$ by trigonometric polynomials and linearity, it suffices to prove the above for $f(x,y) =e^{2 \pi i (n x+m y)}$, with $n,m \in \Z$, and clearly this is nontrivial only when either $n$ or $m$ is different from 0. Since $(q_k)$ strictly increases, the bounded functions 
$e^{2 \pi i (n q_k \, x+m q_k \, y)}$ are pairwise orthogonal (in the Hilbert space $L^2(dx\times dy)$). 
By Rajchman's Strong Law  of Large Numbers (cf. \cite{Ch68}) this implies that 
$$
\lim_N \frac1N \sum_{k=0}^{N-1}e^{2 \pi i (n q_k x+m q_k y)} = 0,
$$
for almost every $(x,y)$, as needed. Since there are countably many pairs $(n,m)$, the equirepartition property holds on a full measure set. In particular, for each pair $(x,y)$ in this set, the \sq\ 
$(q_k x, q_k y) \text{ mod } 1$ is dense in $S_1^2$. By change of coordinates, there exists a set $E$ 
of full Lebesgue measure in $S_1^2$, such that for every pair $(t,t') \in E$ the \sq\ $(q_k (t-t'), q_k (t+t')) \text{ mod } 1$ is dense in $S_1^2$, in particular $(t, t')$ does not satisfy \eqref{minn0}.
The assertion of the theorem is obtained by choosing $(q_k)$ to be the sequence of denominators in the continued fraction expansion of $\alpha$.
\end{proof}


We now pass to the main proof.

\begin{proof}[Proof of Theorem \ref{ergCocy0}] According to Remark \ref{nonsep}, we will focus on proving condition \eqref{condC1}. Condition \eqref{minn0} will be derived at the end.

So, suppose that \eqref{condC1} fails. Then there is $i_0\in\{1,2,3,4\}$ and a sub\sq\ $(k_{\,l})$ (of the indices $k$) and $\gamma>0$ such that for any $i\neq i_0$ in $\{1,2,3,4\}$ the limit below exists and satisfies
\begin{equation}\label{totu}
\lim_l\|q_{k_l}(x_i-x_{i_0})\|>\gamma.
\end{equation}

Denote $\varphi=\varphi_{t,t'}^u$.
The discontinuities of the $n$-step cocycle $\varphi^{(n)}$ occur at the $4n$ points $x_i - j\alpha$, where $i \in \{1,2,3,4\}$, $j\in\{0,1,\dots,n-1\}$. Points obtained for the same index $i$ will be called \emph{discontinuities of type~$i$}. In fact, some of the discontinuity points may coincide (in which case the corresponding discontinuities disappear), but this case will turn out to be trivial. 
\smallskip

By minimality of the irrational rotation, the return times of the orbit of 0 to its neighborhood form a
syndetic set. Since $(q_k)$ grows to infinity, it is easy to find a \sq\ $n_k$ such that $n_k\alpha \to 0$ in $S_1$ and which has the same asymptotics as $\gamma q_k$, i.e., $\frac{n_k}{q_k}\to \gamma$. 

We are assuming that $\varphi$ is not ergodic. Lemma \ref{noSolLem} yields that $\varphi^{(n_k)}$ tends to 1 in $L^1(ds)$. The $L^1$-distance of $\varphi^{(n_k)}$ to the constant function $1$ equals twice the joint measure of the intervals on which $\varphi^{(n_k)}=-1$. First suppose that no discontinuity of type $i_0$ 
coincides with another discontinuity. The function $\varphi^{(n_k)}$ assumes the value $-1$ on exactly one side of each discontinuity of type $i_0$, until the nearest discontinuity point on that side. Thus, the measure of the set where $\varphi^{(n_k)}=-1$ is estimated from below by $n_k$ (the number of discontinuities of type $i_0$) times the minimal distance between such a discontinuity and the nearest discontinuity. If a discontinuity of type $i_0$ coincides with another, this minimal distance equals zero and the above estimate holds trivially. In either case, we conclude that this minimal distance must be of the order $o(\frac1{n_k})$, which is the same as $o(\frac1{q_k})$.

By the remark following \eqref{minmajqn}, for large enough $k$, the above minimal distance must not occur between discontinuities of the same type $i_0$, that is the minimal distance occurs between a discontinuity of type $i_0$ and one of a different type $i\neq i_0$ (depending on $k$). Clearly, as we proceed along the previously selected sub\sq\ $(k_{\,l})$, some type $i_1\neq i_0$ appears in this role infinitely many times. Replacing $(k_{\,l})$ by a sub\sq, we can thus assume that we always have the same (fixed) type $i_1\neq i_0$. Summarizing, we have fixed a sub\sq\ $(k_{\,l})$ of the indices $k$ and a pair of types $i_0$ and $i_1\neq i_0$ such that there exist two \sq s of nonnegative integers $(j_l), (j'_l)$ bounded from above by $(n_{k_l}-1)$, satisfying
$$
\lim_l q_{k_l}\|(x_{i_0} - j_l \alpha) - (x_{i_1} - j'_l \alpha)\| = 0.
$$

Now, since 
$$
\|q_{k_l}(j_l - j'_l)\alpha\| \le |j_l - j'_l| \cdot \|q_{k_l}\alpha\| \le \frac{n_{k_l}}{q_{k_l+1}} \le \frac{n_{k_l}}{q_{k_l}}\underset{l}\longrightarrow \gamma,
$$
(we have used \eqref{pnqn} for the central inequality), we can write
\begin{multline*}
\gamma < \lim_l\|q_{k_l} (x_{i_0} - x_{i_1})\|\le \\
\lim_l q_{k_l}\|(x_{i_0} - j_l \alpha) - (x_{i_1} - j'_l \alpha)\| + \limsup_l\|q_{k_l}(j_l- j'_l) \alpha\| \le 0+ \gamma.
\end{multline*}
This contradiction ends the proof of \eqref{condC1}.

\medskip
To show (\ref{minn0}), it suffices to consider sub\sq s $J=(k_j)$ such that the limits
$$
\beta_J(s) = \lim_j \|q_{k_j}x\|
$$
exist for all the 24 points $x$ obtained by adding or subtracting pairs of different points from $\{x_1,x_2,x_3,x_4\}$ (every sub\sq\ of $(k)$ contains a sub-sub\sq\ $J$ of this kind), and prove 
that for any such $J$ either $\beta_J(t-t') = 0$ or $\beta_J(t+t') = 0$. 

The already proved four convergences in the assertion of the theorem imply that
\begin{align*}
&\min\{\beta_J(t), \beta_J(u), \beta_J(t'+u)\} = 0, \\ 
&\min\{\beta_J(t), \beta_J(t-u), \beta_J(t - t' - u)\} = 0, \\
&\min\{\beta_J(u), \beta_J(t-u), \beta_J(t')\}= 0, \\ 
&\min\{\beta_J(t'+u), \beta_J(t - t' - u), \beta_J(t')\} = 0.
\end{align*}
We will consider three cases:
\begin{itemize}
	\item If $\beta_J(u) = 0$ then $\min\{\beta_J(t), \beta_J(t - t')\} = 0$ (from the second relation) and $\min\{\beta_J(t'), \beta_J(t - t')\} = 0$ (from the fourth relation), therefore either $\beta_J(t - t')=0$
or both $\beta_J(t)=0$ and $\beta_J(t')=0$. But the latter possibility also implies $\beta_J(t - t')=0$.

 \item If $\beta_J(u) > 0$ and $\beta_J(t) = 0$ then $\beta_J(t') = 0$ (from the third relation), which implies $\beta_J(t-t')=0$.
 
 \item If $\beta_J(u) > 0$ and $\beta_J(t) > 0$ then $\beta_J(t'+u) = 0$ (from the first relation) and $\min\{\beta_J(t-u), \beta_J(t')\}= 0$ (from the third relation). In this case, $\beta_J(t')= 0$ is impossible (it would give $\beta_J(t'+u) =\beta_J(u)> 0$), so that we must have $\beta_J(t-u)=0$, 
which combined with $\beta_J(t'+u) = 0$ yields $\beta_J(t+t') = 0$.
\end{itemize}
Therefore in all cases we have either $\beta_J(t-t') = 0$ or $\beta_J(t+t') = 0$, as needed.
\end{proof}

\appendix
\section{Conditional disjointness}

As we know, if two invariant measures are disjoint then they are also uncorrelated.
This follows merely from the fact that the condition defining disjointness is
the same as that for uncorrelation, only it must hold for all, not just one selected
pair of measurable functions. Our Example \ref{example} suggests that there is a weaker version of disjointness, which also implies uncorrelation. This condition, defined below, is specifically well applicable to cocycle extensions.

\begin{defn} Let \xmt\ be a measure-preserving system and let $\mathcal B$ be a
sub-sigma-algebra of $\mathcal A$. We will say that \xmt\ (or shortly $\mu$) is \emph{conditionally}
(given $\mathcal B$) \emph{disjoint} from another system \xmtp\ (or shortly, from $\mu'$) if
$$
\int f(x)\, g(x') \, d\xi = 0,
$$
for every joining $\xi$ of $\mu$ and $\mu'$ and every pair of bounded measurable functions $g$
on $X'$ and $f$ on $X$ satisfying $\E(f|\mathcal B) = 0$.
\end{defn}

Clearly, such an $f$ has zero integral, hence the right hand side above can be written as
$\int f(x) \, d \mu \ \int g(x') \, d \mu'$, as in the definition of (unconditional) disjointness.
Disjointness is the same as conditional disjointness given the trivial sigma-algebra.

The following is now obvious.

\begin{fact}\label{oo}
Suppose \xmt\ is a symbolic system over a finite alphabet $K$ contained
in the unit disc and that it is conditionally (given some sub-sigma-algebra $\mathcal B$)
disjoint from any ergodic system. If the ``first symbol value'' function $f$ satisfies
$\E(f|\mathcal B) = 0$ then $\mu$ is uncorrelated to any ergodic measure.
\end{fact}

The above fact is especially useful in considerations of symbolic representations
of cocycle extensions (of some \yns) with cocycles taking values in finite subgroups
K of the unit circle. In such case, the first symbol value function evaluated
at $(y, \kappa)$ equals $\kappa$ and has zero expectation with respect to $\mathcal B$.
This situation occurs in our Example \ref{example}, where the proof shows that
the measure $\mu$ is conditionally (given the irrational rotation factor)
disjoint from any ergodic measure. Hence the conditions in Fact \ref{oo} are satisfied.

This can be generalized as follows:

Suppose \xmt\ is a symbolic representation of a cocycle extension
of an ergodic discrete spectrum system \yns\, with a cocycle taking values in a
finite subgroup $K$ of the unit circle.  The discrete spectrum system can be represented as an
ergodic rotation of a compact Abelian group (the addition $y+u$, for $y, u \in Y$ is understood 
in this group).

\begin{prop}Suppose the ergodic decomposition of $\mu$ is  $\mu = \int \mu_t \, dt, \ t \in [0, 1],$
(note that this parametrization can be applied whenever the decomposition is nonatomic), 
where each $\mu_t$ also represents a cocycle extension of \yns\ with a cocycle $\phi_t$. If for almost every pair $t, t'$ and every $u \in Y$, the cocycle $\phi^u_{t, t'} (y) = \overline{\phi_t(y)} \, \phi_{t'} (y + u)$ is ergodic, then $\mu$ is conditionally (given $\mathcal B$)
disjoint from any ergodic system, hence, by Fact \ref{oo}, uncorrelated to any ergodic measure.
\end{prop}

We skip the proof which is, up to easy generalizations, implicitly included in Example \ref{example}.

\end{document}